\documentclass{birkjour}
 \newtheorem{thm}{Theorem}[section]
 \newtheorem{cor}[thm]{Corollary}
 \newtheorem{lem}[thm]{Lemma}
 
 \theoremstyle{definition}
 
 \theoremstyle{remark}
 \newtheorem{rem}[thm]{Remark}
 
 \numberwithin{equation}{section}

\begin{document}

%
%
%
%
%
%
%
%
%

\title[Functional identities involving inverses]
 {Functional identities involving inverses on \\Banach algebras}

\author[K. Luo]{Kaijia Luo}

\address{%
Institute of Mathematics\\
Hangzhou Dianzi University\\
Hangzhou\\
China}
\email{kaijia\_luo@163.com}
\thanks{Corresponding author: Kaijia Luo. E-mail addresses: kaijia\_luo@163.com.\\
This work was supported by  the National Natural Science Foundation of China (Grant No.11871021).}
\author[J. Li]{Jiankui Li}
\address{School of Mathematics\br
East China University of Science and Technology\br
Shanghai\br
China}
\email{jkli@ecust.edu.cn}


\date{}

\begin{abstract}
The purpose of this  paper is to characterize several classes of functional identities involving inverses with related mappings from a  unital Banach algebra $\mathcal{A}$ over the complex field into  a  unital $\mathcal{A}$-bimodule $\mathcal{M}$.  Let $N$ be a fixed invertible element in $\mathcal{A}$, $M$ be a fixed element in $\mathcal{M}$, and $n$ be a positive integer.
We investigate the forms of additive mappings $f$, $g$ from $\mathcal{A}$ into $\mathcal{M}$ satisfying one of the following identities:
\begin{equation*}
\begin{aligned}
&f(A)A- Ag(A) = 0\\
&f(A)+ g(B)\star A= M\\
&f(A)+A^{n}g(A^{-1})=0\\
&f(A)+A^{n}g(B)=M
\end{aligned}
\qquad
\begin{aligned}
&\text{for each invertible element}~A\in\mathcal{A}; \\
&\text{whenever}~ A,B\in\mathcal{A}~\text{with}~AB=N;\\
&\text{for each invertible element}~A\in\mathcal{A}; \\
&\text{whenever}~ A,B\in\mathcal{A}~\text{with}~AB=N,
\end{aligned}
\end{equation*}
where  $\star$ is either the Jordan product $A\star B = AB+BA$ or the Lie product $A\star B = AB-BA$.
\end{abstract}

\maketitle
\section{Introduction}
The theory of functional identities deals with the form or  structure of functions on rings  that satisfy certain  identities involving arbitrary or special elements.
Functional identities have been demonstrated to be applicable in solving  certain problems in  mathematical fields, especially in operator theory and functional analysis.
A comprehensive introduction to the theory of  functional identities and their applications is given in the book \cite{BCM}.

Let $\mathcal{R}$ be a ring. The identity $A f(B) + B g( A) = 0$ for all $A, B \in \mathcal{R}$ with the functions $f$, $g:\mathcal{R}\rightarrow\mathcal{R}$ is a very basic example of  functional identities.
 It is obvious that a pair of functions of the form $f=g=0$ is a standard solution of the above functional identity.
The \emph{standard solutions} of functional identities are those that do not depend on some structural property of the ring, but are merely consequences of ring axioms and formal computations.

The exact beginning of research on results regarding functional identities involving inverses on rings  is difficult to determine, due to the wide range of inclusion of functional identities. 
A notable example of functional identities is the commuting mapping.
 A mapping $f$ from a  ring $\mathcal{R}$ into itself is called  \emph{commuting}  if $f(A)A = Af(A)$ for each element $A\in\mathcal{R}$.
An important result on commuting mappings is Posner’s theorem \cite{P}  published in 1957. This theorem shows that the existence of a nonzero commuting derivation on a  prime ring $\mathcal{R}$ implies that $\mathcal{R}$ is commutative.
It was proved in \cite[Theorem 3.2]{B1993} that if  $f$ is an additive commuting mapping on a prime ring $\mathcal{R}$, there exist an element $\lambda$ in the extended centroid $\mathcal{C}$ of $\mathcal{R}$ and an additive
mapping $\mu:\mathcal{R}\rightarrow\mathcal{C}$ such that $f(A)=\lambda A+\mu(A)$ for each $A\in\mathcal{R}$.
The further  comprehensive discussion  about commuting mappings can be found in a survey paper \cite{B}.
In 1995,   Bre\v{s}ar \cite{B1995}  investigated   additive mappings $f, g$ from a prime ring $\mathcal{R}$ into itself satisfying the identity 
\begin{equation*}\label{1}
f(A)A- Ag(A) = 0\tag{F1}
\end{equation*} for each $A\in\mathcal{R}$ which is a generalization of the commuting mapping.
Recently, Ferreira, Julius,  and  Smigly showed   that if  additive mappings $f,g$ from an alternative division ring $\mathcal{D}$  into itself satisfy the  identity (\ref{1}) for each $A$ in $\mathcal{D}$, there exist a fixed element $\lambda\in\mathcal{D}$ and an additive mapping $\mu$ from $\mathcal{D}$ into the  centre of $\mathcal{D}$ such that $f(A) = A\lambda +\mu(A)$ and $g(A) =\lambda A + \mu(A)$ for each  $A\in\mathcal{D}$ (see \cite[Theorem 4.4]{FJS}).
Even more,  a special class of functional identities introduced  in \cite{FJS} is that two additive mappings $f,g$ on $\mathcal{D}$ satisfy the  identity 
 \begin{equation*}\label{2}
f(A)+ g(B)\star A= M\quad \text{whenever}~AB=N,\tag{F2}
\end{equation*} 
where $M,N$ are  fixed elements in $\mathcal{D}$ and $\star$ is either the ordinary product $A\star B = AB$, the Jordan product $A\star B = AB+BA$, or the Lie product $A\star B = AB-BA$.

As early as 1987, Vukman in \cite{V1987} showed that if an additive mapping $f$ on a noncommutative division ring  $\mathcal{F}$ of characteristic not 2  satisfies the identity $f(A^{-1})+A^{2}f(A^{-1})=0$ for each invertible element $A$ in $\mathcal{F}$, then the identity has only the standard solution $f = 0$. 
Vukman's identity  is a special class of the identity  $F(x)+ M(x)G(1/x)=0$ for each invertible element $x$  where $F, G$ are additive and $M$ is multiplicative, and the  latter  identity is characterized on a commutative field of characteristic not 2 by Ng in \cite{N}.
Especially worth mentioning, the study of the  above  identities gives a motivation to study  the following functional identity:
additive mappings $f,g$ from a ring  $\mathcal{R}$ into $\mathcal{R}$ satisfying the identity
\begin{equation*}\label{3}
 f(A)+A^{n}g(A^{-1})=0 \tag{F3}
\end{equation*}
 for all invertible element $A\in \mathcal{R}$, where $n$ is a positive integer.
Obviously, the solution $f=g=0$ is the standard solution of  the  identity (\ref{3}).
Catalano and  Merchán  proved that  additive mappings $f,g$ satisfying the  identity (\ref{3}) for $n=3$ or $n=4$ on a  division ring of characteristic not 2,~3  are of the form $f=g=0$ (see \cite[Theorem 1]{CM}).
Continuing this line of investigation, it was proved in \cite{LCW}  that when $\mathcal{D}$ is a division ring with $char(\mathcal{D})>n$ or $char(\mathcal{D})=0$,  the  identity (\ref{3}) on  $M_{m}(\mathcal{D})~(m \geq 3)$ has only the  standard solution.
Lee and  Lin proved in \cite{LeeL} that the only solution of  the identity (\ref{3}) for $n \neq2$ on a noncommutative division ring of characteristic not 2 is the standard solution.

Throughout this paper, let $\mathcal{A}$ be a unital Banach algebra over the complex field $\mathbb{C}$ and $\mathcal{M}$ be a unital $\mathcal{A}$-bimodule. The properties of Banach algebras  play a crucial role in the study of the  identities.  Without specified otherwise,  we assume that all  algebras are over  $\mathbb{C}$. 
We also discuss the following identity:
 \begin{equation*}\label{4}
f(A)+A^{n}g(B)=M\quad \text{whenever}~AB=N,\tag{F4}
 \end{equation*} 
with additive mappings $f$, $g:\mathcal{A}\rightarrow \mathcal{M}$ for fixed elements  $M\in\mathcal{M}$, $N\in \mathcal{A}$.

The paper is organized as follows: 
 Section 2 describes the additive mappings $f,g$ from $\mathcal{A}$ into  $\mathcal{M}$  satisfying the identity (\ref{1}) for all invertible elements but not any element or the identity (\ref{2}) when $\star$ is the Jordan product or the Lie product. 
 Section 3 is devoted to characterizing the additive mappings $f,g$ from $\mathcal{A}$ into  $\mathcal{M}$  satisfying the  identity (\ref{4})
  where  $N$ is an arbitrary given element in $\mathcal{A}$ and $M$ is a fixed element in  $\mathcal{M}$. 
The  process  of investigation involves the discussion of the  identity (\ref{3}).

\section{Identities (\ref{1}) and (\ref{2})}

In this  section, we discuss  the  identities (\ref{1}) and (\ref{2}) with  related mappings from $\mathcal{A}$ into $\mathcal{M}$.
In the sequel we  use $[A, B]$ to denote the Lie product $AB - BA$, and $A \circ B$ to denote the Jordan product $AB + BA$. Denote the center of a ring $\mathcal{R}$ by $Z(\mathcal{R})$. 
In particular, the  definition of the commuting mapping is relaxed as follows: a  mapping $f$ from a ring $\mathcal{R}$ into its bimodule is called a \emph{commuting mapping} if  $[A, f(A)]=0$ for each $A$ in $\mathcal{R}$.

\begin{thm}\label{t3.1}
 If two additive mappings $f$, $g:\mathcal{A}\rightarrow \mathcal{M}$ satisfy the identity
 $$f(A)A - Ag(A)=0$$
 for each invertible element $A$ in $\mathcal{A}$, then there is an additive commuting mapping $h$ such that $f(A)=h(A)+Af(I)$ and $g(A)=h(A)+f(I)A$ for each  $A\in\mathcal{A}$.
 Moreover, $f=g$  if and only if $f$ is a commuting mapping. 
\end{thm}
\begin{proof}
It is obvious that $f(I)-g(I)=0$.
Let $T\in\mathcal{A}$, $n \in \mathbb{N}$ with $n> \|T\|$, and $A=nI+T$. Then $A$ and $I+A$ are both invertible in $\mathcal{A}$.
By assumption, we have 
\begin{equation}\label{3.0}
f(A)A - Ag(A)=0
\end{equation} and 
\begin{equation}\label{3.1}
 f(I+A)(I+A)-(I+A)g(I+A)=0. 
\end{equation}
By simplifying the equation (\ref{3.1}), we  obtain
 $f(A)-g(A)=Ag(I)-f(I)A.$
By the additivity of $f$, it follows that 
\begin{equation}\label{3.2}
 f(T)-g(T)=Tg(I)-f(I)T=[T,f(I)]
\end{equation}
for any element $T\in\mathcal{A}$.
Putting $A=nI+T$ into the equation (\ref{3.0}) gives
$$f(nI+T)(nI+T)-(nI+T)g(nI+T)=0.$$
Apply the equation (\ref{3.2}) and $f(I)=g(I)$ in the above equation to get
\begin{align*}
0 &= f(nI+T)(nI+T)-(nI+T)g(nI+T)\\
 &= n^{2}(f(I)-g(I))+n(f(I)T-Tg(I))+n(f(T)-g(T))+f(T)T-Tg(T)\\
 &= n(f(I)T-Tf(I))+n(f(T)-g(T))+f(T)T-Tg(T)\\
 &= n(f(I)T-Tf(I))+n[T,f(I)]+f(T)T-Tg(T)\\
 &= f(T)T - Tg(T).
\end{align*}
It means that 
\begin{equation}\label{3.3}
f(T)T =Tg(T)
\end{equation}
for each element $T\in\mathcal{A}$.
On the one hand, we have
\begin{equation}\label{3.4}
(f(T)-g(T))T=f(T)T -g(T)T=Tg(T)-g(T)T=[T,g(T)].
\end{equation}
On the other hand,  using the  equation (\ref{3.2}), we obtain
\begin{equation}\label{3.5}
(f(T)-g(T))T=[T,f(I)]T=[T,f(I)T].
\end{equation}  
Comparing the equations (\ref{3.4}) and (\ref{3.5}), we arrive at 
$[T,g(T)]=[T,f(I)T]$. 
It means that $[T,g(T)-f(I)T]=0$ for each $T\in\mathcal{A}$.
Thus $g(T)-f(I)T$ is an additive  commuting mapping.
Therefore there is an additive  commuting mapping $h$ such that $$g(T)=h(T)+f(I)T$$ for each  $T\in\mathcal{A}$.
It follows that 
$$f(T)=g(T)+[T,f(I)]=h(T)+f(I)T+[T,f(I)]=h(T)+Tf(I)$$
for each  $T$ in $\mathcal{A}$.

 Moreover,  if $f=g$, it follows that  $Af(I)=f(I)A$ for each  $A$ in $\mathcal{A}$. It is obvious that  $f$ is a commuting mapping. 
  On the contrary, since $f(A)A=Af(A)$ for each  $A$ in $\mathcal{A}$, we have $(h(A)+Af(I))A=A(h(A)+Af(I))$, which implies 
  $Af(I)A=A^{2}f(I)$. Replace $A$ with $A+I$ to get $f(I)A=Af(I)$ for each  $A$ in $\mathcal{A}$. Therefore, we have $f=g$.

\end{proof}

The standard solution of a commuting  mapping $f$ on a ring $\mathcal{R}$ with the centre  $Z(\mathcal{R})$ is defined by
\begin{equation}\label{3.11}
f(A)=\lambda A+\mu(A)
\end{equation}
 for every $A\in\mathcal{R}$ where $\lambda \in Z(\mathcal{R})$ and $\mu$ is a mapping from $\mathcal{R}$ into $Z(\mathcal{R})$.
A commuting mapping of the form (\ref{3.11}) is said to be \emph{proper}.
Sometimes, the standard solution is also the unique solution on specific rings or algebras, such as unital simple  rings \cite{B} and  full matrix algebras \cite{XW}.

\begin{lem}\cite[Corollary 3.3]{B}\label{lem3.1}
Let $\mathcal{R}$ be a unital simple  ring. Then every additive commuting mapping  on $\mathcal{R}$ is 
proper.
\end{lem}

\begin{lem}\cite[Corollary 4.1]{XW}\label{lem3.2}
Let $\mathcal{R}$ be a unital commutative ring  and $\mathcal{B}$ be a unital algebra over $\mathcal{R}$. 
Any $\mathcal{R}$-linear commuting  mapping on the full matrix algebra $M_{n}(\mathcal{B})$ over $\mathcal{B}$ is proper, where $n\geq2$.
\end{lem}

Combining Theorem \ref{t3.1}, Lemmas \ref{lem3.1}  and  \ref{lem3.2}, we can directly derive the following corollaries.

\begin{cor}\label{cor3.1}
 Let $\mathcal{B}$ be a unital simple Banach algebra.
 If two additive mappings $f$, $g:\mathcal{B}\rightarrow \mathcal{B}$ satisfy the identity
 $f(A)A - Ag(A)=0$
 for each invertible element $A$ in $\mathcal{B}$, then there exist a fixed element $Q\in\mathcal{B}$ and  an  additive mapping $h:\mathcal{B} \rightarrow Z(\mathcal{B})$ such that $f(A)=h(A)+AQ$ and $g(A)=h(A)+QA$ for each  $A\in\mathcal{B}$.
\end{cor}
\begin{proof}
  It follows from  Theorem \ref{t3.1} that there is an additive commuting mapping $h$ such that $f(T)=h(T)+Tf(I)$ and $g(T)=h(T)+f(I)T$ for each  $T\in\mathcal{B}$.
  In view of Lemma \ref{lem3.1}, we have $h(T)=\lambda T+\mu(T)$  for each $T\in\mathcal{B}$,
where $\lambda \in Z(\mathcal{B})$ and $\mu$ is an additive mapping from $\mathcal{B}$ into $Z(\mathcal{B})$.
Thus we conclude that  $$f(T)=T(f(I)+\lambda)+\mu(T),~g(T)=(f(I)+\lambda)T+h(T)$$
 for each $T\in\mathcal{B}$.
Therefore,  taking $Q=f(I)+\lambda$ completes the proof.
\end{proof}

Using the same proof  of Corollary \ref{cor3.1}, the following conclusion can be obtained.
\begin{cor}\label{cor3.2}
 Let $\mathcal{B}$ be a unital algebra.
 If two linear mappings $f$, $g:M_{n}(\mathcal{B})\rightarrow M_{n}(\mathcal{B})$ $(n\geq2)$ satisfy the identity
 $f(A)A - Ag(A)=0$
 for each invertible element $A$ in $M_{n}(\mathcal{B})$, then there exist a fixed element $Q\in M_{n}(\mathcal{B})$ and  a linear mapping $h$ from $M_{n}(\mathcal{B})$ into its center such that $f(A)=h(A)+AQ$ and $g(A)=h(A)+QA$ for each  $A\in M_{n}(\mathcal{B})$.
\end{cor}

In the last part of  the section, we study the identity (\ref{2}) with  related mappings from  a unital  Banach algebra $\mathcal{A}$ into  a unital $\mathcal{A}$-bimodule $\mathcal{M}$ when $\star$ is  the Jordan product or the Lie product.
The  identity (\ref{2}) for  the ordinary product is not discussed here because it is similar to the identity (\ref{3}) which is characterized in the next section.
Using Theorem  \ref{t3.1},  we give a simplified version of the identity (\ref{2})  defined invertible elements under the Jordan product operation.
\begin{lem}\label{lem3.3}
   Suppose that $M$ is a fixed element in  $\mathcal{M}$.
 If two additive mappings $f$, $g:\mathcal{A}\rightarrow \mathcal{M}$ satisfy the identity
$$f(A)+g(A^{-1})\circ A=M$$   for each invertible element $A$ in $\mathcal{A}$,
then $f=0$ and $g(A)=\frac{1}{2}(A g(I)+g(I)A)$ for each  $A$ in $\mathcal{A}$. 
\end{lem}
\begin{proof}
For each invertible element $A$ in $\mathcal{A}$, by assumption,  we have 
 $$f(-A)+g(-A^{-1})\circ (-A)=M.$$
 Thus $f(A)=0$ and   $g(A^{-1})\circ A=M$.
It follows that  $M=2g(I)$ and  
\begin{equation}\label{4.0}
Ag(A^{-1})+(g(A^{-1})-MA^{-1}) A=0.
\end{equation}
Define an additive mapping $\tilde{g}$ from $\mathcal{A}$ into $\mathcal{M}$ by $\tilde{g}(T)=g(T)-MT$ for every $T\in\mathcal{A}$.
Then the equation (\ref{4.0}) can be rewritten as
$$Ag(A^{-1})+\tilde{g}(A^{-1}) A=0.$$
Replacing $A$ with $A^{-1}$ in the above equation gives 
$A^{-1}g(A)+\tilde{g}(A) A^{-1}=0,$
i.e., $g(A)A+A\tilde{g}(A) =0$ for each invertible element $A$ in $\mathcal{A}$.
By Theorem \ref{t3.1}, there is an additive commuting mapping $h$ such that 
\begin{equation}\label{4.1}
  g(T)=h(T)+T g(I)
\end{equation} and $-\tilde{g}(T)=h(T)+g(I)T$.
From  the definition of $\tilde{g}$, we have $-(g(T)-MT)=h(T)+g(I)T$.
It means that  
\begin{equation}\label{4.2}
 g(T)=-h(T)-g(I)T+MT=-h(T)+g(I)T
\end{equation}
for each $T$ in  $\mathcal{A}$.
Applying the equations (\ref{4.1}) and  (\ref{4.2}), we have 
 $$g(T)=\frac{1}{2}(2g(T))=\frac{1}{2}(h(T)+T g(I)-h(T)+g(I)T)=\frac{1}{2}(T g(I)+g(I)T)$$
  for each $T$ in $\mathcal{A}$.
\end{proof}

\begin{thm}
  Suppose that $N$ is an invertible element in $\mathcal{A}$, and $M$ is a fixed element in  $\mathcal{M}$. 
 If two additive mappings $f$, $g:\mathcal{A}\rightarrow \mathcal{M}$ satisfy the identity
$$f(A)+g(B)\circ A=M$$  whenever $A,B\in\mathcal{A}$ with $AB=N$, 
then $f=0$ and $g(T)=\frac{1}{2}(AN^{-1}g(I)+g(I)AN^{-1})$ for each  $A\in\mathcal{A}$. 
\end{thm}
\begin{proof}
For each invertible element $A$ in $\mathcal{A}$, it follows from  $A(A^{-1}N)=N$ that
$$f(A)+g(A^{-1}N)\circ A=M.$$
By Lemma \ref{lem3.3}, we have $f=0$ and $g(AN)=\frac{1}{2}(A g(I)+g(I)A)$ for each  $A$ in $\mathcal{A}$.
Since $N$ is an invertible element in $\mathcal{A}$, we obtain  $$g(A)=\frac{1}{2}(AN^{-1}g(I)+g(I)AN^{-1})$$ for each  $A$ in $\mathcal{A}$
\end{proof}

We consider the Lie product version of Lemma \ref{lem3.3}.

\begin{lem}
   Suppose that $M$ is a fixed element in  $\mathcal{M}$.
 If two additive mappings $f$, $g:\mathcal{A}\rightarrow \mathcal{M}$ satisfy the identity
$$f(A)+[g(A^{-1}),A]=M$$ for each invertible element $A$ in $\mathcal{A}$,
then $f=0$ and $g$ is an additive commuting mapping.
\end{lem}
\begin{proof}
Using invertible elements  $A$ and $-A$, we have $f=0$ and $[g(A^{-1}),A]=M$ for each invertible element $A$ in $\mathcal{A}$.
  Thus $M=[g(I),I]=0$.
  It follows that  $[g(A),A^{-1}]=0$, which  is equivalent to $[g(A),A]=0.$
  By Theorem \ref{t3.1}, $g$ is an additive commuting mapping.
\end{proof}

\begin{thm}
  Suppose that $N$ is an invertible element in $\mathcal{A}$, and $M$ is a fixed element in  $\mathcal{M}$. 
 If two additive mappings $f$, $g:\mathcal{A}\rightarrow \mathcal{M}$ satisfy the identity
$$f(A)+[g(B),A]=M$$  whenever $A,B\in\mathcal{A}$ with $AB=N$, 
then $f=0$ and there is an additive commuting mapping $h$ such that $g(A)=h(A)+Ag(I)$ for each  $A\in\mathcal{A}$.
\end{thm}
\begin{proof}
Since  $N=N A^{-1}A=(-N A^{-1})(-A)$, we have $f=0$ and $[g(A),N A^{-1}]=M$ for each invertible element $A$ in $\mathcal{A}$.
Multiplying $N^{-1}$ from the left side  of the above identity, we obtain $(N^{-1} g(A)N-N^{-1}MA)A^{-1}-A^{-1}g(A)=0,$ i.e.,
$$g(A)A-A(N^{-1} g(A)N-N^{-1}MA)=0.$$
By Theorem \ref{t3.1}, there is an additive commuting mapping $h$ such that $g(A)=h(A)+Ag(I)$ for each $A\in\mathcal{A}$.
\end{proof}

\section{Identities (\ref{3}) and (\ref{4})}

The main results of this section give a characterization of two additive mappings $f$, $g:\mathcal{A}\rightarrow \mathcal{M}$ satisfying the identity (\ref{4}).
We divide into three cases $n=1$, $n=2$ and $n>2$ to consider the identity (\ref{4}). Besides, every case involves the study of the identity (\ref{3}).

 The discussion begins with the  identity (\ref{3}) for $n=1$.
\begin{lem}\label{lem2.1}
 If two additive mappings $f$, $g:\mathcal{A}\rightarrow \mathcal{M}$ satisfy the identity
 $$f(A)+Ag(A^{-1})=0$$
 for each invertible element $A$ in $\mathcal{A}$, then $f=g=0$.
\end{lem}
\begin{proof}
Suppose that $A$ is an invertible element in  $\mathcal{A}$.
Since $-A$ is also  invertible in  $\mathcal{A}$, by assumption, we have 
 $f(A)+Ag(A^{-1})=0$ and $-f(A)+Ag(A^{-1})=0$.
 Comparing these two equations, we obtain  $f(A)=0$ and $Ag(A^{-1})=0$.
 It follows that  $g(A)=0$ for each  invertible element  $A$ in  $\mathcal{A}$.
 In particular, we have $f(I)=g(I)=0$.

Let $T\in\mathcal{A}$, $n \in \mathbb{N}$ with $n> \|T\|+1$. 
Note that $nI+T$ is invertible in $\mathcal{A}$.
Then we arrive at $f(nI+T)=g(nI+T)=0$. 
The additivity of $f$ and $g$ implies $f(T)=g(T)=0$  for each element $T$ in $\mathcal{A}$.
\end{proof}

From the last paragraph of the  proof of Lemma \ref{lem2.1} , we can observe the following result.
\begin{rem}\label{rem1}
An additive mapping $f$ from a unital Banach  algebra  into its module satisfies that $f(A)=0$ for any invertible element $A$, then $f$ is identically equal to 0.
\end{rem}
As given in the theorem below, using  Lemma \ref{lem2.1}, additive mappings  satisfying the  identity (\ref{4}) for $n=1$ have a specific structure.

\begin{thm}\label{th2.1}
 Suppose that $N$ is an invertible element in $\mathcal{A}$ and $M$ is a fixed element in  $\mathcal{M}$.
 If two additive mappings $f$, $g:\mathcal{A}\rightarrow \mathcal{M}$ satisfy the identity
 $$f(A)+Ag(B)=M$$ 
 whenever $A,B\in\mathcal{A}$ with $AB=N$,   then $f=0$ and   $g(A)=AN^{-1}M$ for each $A\in\mathcal{A}$.
\end{thm}
\begin{proof}
For each invertible element $A\in\mathcal{A}$, it follows from $AA^{-1}N=N$ that $f(A)+Ag(A^{-1}N)=M$, which means that  
$$f(A)+A(g(A^{-1}N)-A^{-1}M)=0.$$
By Lemma \ref{lem2.1}, $f=0$ and $g(AN)-AM=0$ for each $A\in\mathcal{A}$.
Replacing $A$ by $AN^{-1}$  gives $$g(A)=AN^{-1}M$$ for each $A\in\mathcal{A}$.
\end{proof}

Different with the proof of Lemma \ref{lem2.1},  our strategy for dealing with the identity (\ref{3})  for  $n = 2$ is to  transform this identity involving two mappings into two equations  in which there is only one unknown mapping in each equation.
A characterization of an additive mapping satisfying the identity (\ref{3}) for  $n = 2$ is actually equivalent to considering two related mapping $f_{1}$ and $f_{2}$ such that $f_{1}(A)+A^{2}f_{1}(A^{-1})=0$ and $f_{2}(A)-A^{2}f_{2}(A^{-1})=0$. 
In the following, we introduce this transformation method, but before that,  the structure of the mappings $f_{1}$ and $f_{2}$ is given.

Recall that an additive mapping $f$ from a ring $\mathcal{R}$ into an $\mathcal{R}$-left module is a \emph{Jordan left derivation} if  $f(A^2)=2f(A)$ holds for each $A\in \mathcal{R}$.
The equation $f(A)+A^{2}f(A^{-1})=0$ is Vukman's identity introduced in Section 1. 
In the following, we consider Vukman's identity with the mapping $f$ from $\mathcal{A}$ into  $\mathcal{M}$.

\begin{lem}\label{lem2.2}
 If an additive mapping $f:\mathcal{A}\rightarrow \mathcal{M}$ satisfies the identity
 $f(A)+A^{2}f(A^{-1})=0$
 for each invertible element $A$ in $\mathcal{A}$, then  $f$ is a Jordan left derivation.
\end{lem}
The proof of Lemma \ref{lem2.2} can refer to  proof of Theorem 5 in \cite{V}. So we omit the proof.

Next we characterize the equation $f_{2}(A)-A^{2}f_{2}(A^{-1})=0$ with  a concise  proof. 
\begin{lem}\label{lem2.3}
 If an additive mapping $f:\mathcal{A}\rightarrow \mathcal{M}$ satisfies the identity
 $f(A)-A^{2}f(A^{-1})=0$
 for each invertible element $A$ in $\mathcal{A}$, then  $f(A)=Af(I)$ for each $A\in \mathcal{A}$.
\end{lem}
\begin{proof}
Assume that $f(I)=0$.
Let $T\in \mathcal{A}$,  $n$ be a positive number with $n>\|T\|+1$, and $B=nI+T$. Then $B$ and $I-B$ are both invertible in $\mathcal{A}$.
By assumption,
\begin{align*}
f(B)&=B^{2}f(B^{-1})=B^{2}f(B^{-1}(I-B))\\
&=B^{2}B^{-2}(I-B)^{2}f(B(I-B)^{-1})\\
&=(I-B)^{2}f(I-B)^{-1}\\
&=(I-B)^{-2}(I-B)^{2}f(I-B)\\
&=-f(B).
\end{align*}
It means that $f(B)=0$.
It follows from $B=nI+T$ and $f(I)=0$  that  $f(T)=0=Tf(I)$ holds for each $T\in \mathcal{A}$.

For the general case of $f(I)\neq0$, define an additive mapping $\tilde{f}$ from $\mathcal{A}$ to $\mathcal{M}$ by
$$\tilde{f}(T)=f(T)-Tf(I)$$
for each $T$ in $\mathcal{A}$.
Then we have  $\tilde{f}(A)-A^{2}\tilde{f}(A^{-1})=f(A)-A^{2}f(A^{-1})=0$ for each invertible element $A$ in $\mathcal{A}$ and $\tilde{f}(I)=0$.
From the previous discussion of this proof,  we have  $\tilde{f}(T)=0$ for each $T\in \mathcal{A}$.
We arrive at $f(T)=Tf(I)$ for each $T\in \mathcal{A}$.
The proof  is complete.
\end{proof}

Based on Lemmas \ref{lem2.2} and \ref{lem2.3}, we consider the identity (\ref{3}) for  $n = 2$ in the following conclusion, which is proved by transforming this identity into  equations in the above lemmas.

\begin{lem}\label{lem2.4}
If two additive mappings $f$, $g:\mathcal{A}\rightarrow \mathcal{M}$ satisfy the identity
 $$f(A)+A^{2}g(A^{-1})=0$$
 for each invertible element $A$ in $\mathcal{A}$, 
 then there exists a  Jordan left derivation $d$ such that $f(A)=d(A)+Af(I)$ and $g(A)=d(A)+Ag(I)$ for each $A\in \mathcal{A}$, where $f(I)=-g(I)$. 
\end{lem}
\begin{proof}
By assumption, we have 
\begin{equation}\label{2.1}
f(A)+A^{2}g(A^{-1})=0 
\end{equation} for each invertible element $A$ in $\mathcal{A}$.
Taking  $A=I$ in the  equation (\ref{2.1}), we obtain 
we have $f(I)+g(I)=0$. 
Replacing  $A$ with $A^{-1}$ in the equation (\ref{2.1}), we obtain $f(A^{-1})+A^{-2}g(A)=0$, which leads to 
\begin{equation}\label{2.2}
A^{2}f(A^{-1})+g(A)=0.
\end{equation}
Adding the equations (\ref{2.1}) and (\ref{2.2}) gives 
$$(f+g)(A)+A^{2}(f+g)(A^{-1})=0.$$
In view of Lemma \ref{lem2.2}, there exists a  Jordan left derivation $d$ such that $f+g=d$.
Subtracting the equation (\ref{2.2}) from the equation (\ref{2.1}), we have 
 $$(f-g)(A)-A^{2}(f-g)(A^{-1})=0.$$
It follows from  Lemma \ref{lem2.3} that $(f-g)(T)=T(f-g)(I)$ for each $T\in \mathcal{A}$.
Therefore, we conclude that
$$f(T)=\frac{1}{2}(f+g+f-g)=\frac{1}{2}(d(T)+T(f-g)(I))=\frac{1}{2}d(T)+Tf(I)$$ and 
$$g(T)=\frac{1}{2}[f+g-(f-g)]=\frac{1}{2}(d(T)-T(f-g)(I))=\frac{1}{2}d(T)+Tg(I)$$ for each $T\in \mathcal{A}$.
Taking $d'=\frac{1}{2}d$, the proof is complete.
\end{proof}

For  the identity (\ref{3}) with $n = 2$,  Bre\v{s}ar and Vukman gave an adequate description of $f$ and  $g$ on a class of algebras in \cite{BV}. 
Precisely, in case $B(\mathcal{X})$ is the algebra of all bounded linear operators on a Banach space $\mathcal{X}$, if additive mappings  $f$,  $g$ which map $B(\mathcal{X})$ into  $\mathcal{X}$ or into $B(\mathcal{X})$ satisfy $f(A)-A^{2}g(A^{-1})=0$ for each invertible operator $A\in B(\mathcal{X})$, then $f$ and $g$ are of the form $f(A)=-g(A)=Af(I)$ for each $A\in B(\mathcal{X})$.
Comparing the result in \cite{BV} with Lemma \ref{lem2.4},  the forms of the mappings $f$ and  $g$  differ only in a Jordan left derivation.
This difference arises from the fact that every Jordan left derivation from $B(\mathcal{X})$ into  $\mathcal{X}$ or into $B(\mathcal{X})$ is equal to zero (see \cite[Corollary 1.5]{BV}).

The question of on which algebras  every Jordan left derivation is identical to zero has attracted the attention of several authors. 
Vukman gave a positive answer for  semisimple Banach algebras,  which states that a linear Jordan left derivation on a semisimple Banach algebra is zero (see \cite[Theorem 4]{V}). 
It was proved that every linear Jordan left derivation from a  $C^*$-algebra into  its Banach left module is zero (see \cite[Theorem 2.5]{ADL}).
One can obtain directly the following corollaries  along Lemma \ref{lem2.4} and the results in \cite{ADL,V}.

\begin{cor}\label{cor2.1}
Let $\mathcal{B}$ be  a unital   semisimple Banach algebra. 
  If two linear  mappings $f$, $g:\mathcal{B}\rightarrow \mathcal{B}$ satisfy the identity
 $f(A)+A^{2}g(A^{-1})=0$
 for each invertible element $A$ in $\mathcal{B}$, 
 then  $f(A)=-g(A)=Af(I)$  for each $A\in \mathcal{B}$. 
\end{cor}

\begin{cor}\label{cor2.2}
Suppose that  two  linear mappings $f,g$ from a unital  $C^*$-algebra $\mathcal{B}$ into its unital Banach bimodule satisfy 
 $f(A)+A^{2}g(A^{-1})=0$
 for each invertible element $A$ in $\mathcal{B}$. 
 Then  $f(A)=-g(A)=Af(I)$  for each $A\in \mathcal{B}$. 
\end{cor}

With the same proof as that of Theorem \ref{th2.1}, the following result is a characterization of  the  identity (\ref{4})  for the case $n = 2$. 
\begin{thm}\label{th2.2.1}
 Suppose that $N$ is an invertible element in $\mathcal{A}$, and $M$ is a fixed element in  $\mathcal{M}$.
 If two additive mappings $f$, $g:\mathcal{A}\rightarrow \mathcal{M}$ satisfy the identity
 $$f(A)+A^{2}g(B)=M \quad \text{whenever}~A,B\in\mathcal{A}~\text{with}~AB=N,$$
 then there exists a  Jordan left derivation $d$ such that $f(A)=d(A)+Af(I)$ and $g(A)=d(AN^{-1})+AN^{-1}g(I)$ for each $A\in \mathcal{A}$. 
\end{thm}

If we multiply $A^{-1}$ from the left side  of the identity (\ref{3})  for  $n = 2$, one can observe the equivalent expression of the identity (\ref{3}), namely the equation $A^{-1}f(A)+Ag(A^{-1})=0$.
We can  extend the above equation to the identity  $Bf(A)+Ag(B)=M$ whenever $A,B\in\mathcal{A}$ with $ AB=N$, which is  characterized in the following. 
\begin{thm}\label{th2.2.2}
 Suppose that $N$ is an invertible element in $\mathcal{A}$, and $M$ is a fixed element in  $\mathcal{M}$.
 If two additive mappings $f$, $g:\mathcal{A}\rightarrow \mathcal{M}$ satisfy the identity
 $$Bf(A)+Ag(B)=M\quad \text{whenever}~A,B\in\mathcal{A}~\text{with}~AB=N,$$
 then there exists a  Jordan left derivation $d$ such that $f(A)=N^{-1}d(A)+N^{-1}ANf(I)$ and $g(A)=d(AN^{-1})+AN^{-1}g(N)$ for each $A\in \mathcal{A}$. 
\end{thm}
\begin{proof}
For each invertible element $A$ in $\mathcal{A}$, it follows from  $A(A^{-1}N)=N$ that  $A^{-1}Nf(A)+Ag(A^{-1}N)=M$, i.e.,
 $$A^{-1}(Nf(A)-AM)+Ag(A^{-1}N)=0.$$  
According to Lemma \ref{lem2.4}, there exists a  Jordan left derivation $d$ such that 
\begin{equation}\label{2.3}
  Nf(T)-TM=d(T)+T(Nf(I)-M)
\end{equation}
and 
\begin{equation}\label{2.4}
g(TN)=d(T)+Tg(N)
\end{equation}
 for each $T\in \mathcal{A}$.
Multiplying $N^{-1}$ from the left side  of the equation (\ref{2.3}), we arrive at 
$$f(T)=N^{-1}d(T)+N^{-1}TNf(I)$$ for each $T\in \mathcal{A}$.
Substituting $TN^{-1}$ for $T$ in the equation(\ref{2.4}) implies 
$g(T)=d(TN^{-1})+TN^{-1}g(N)$ for each $T\in \mathcal{A}$.
\end{proof}

Now we consider the identity (\ref{3}) for the integer $n>2$.

\begin{lem}\label{lem2.5}
 Let $n >2$ be an integer.
If two additive mappings $f$, $g:\mathcal{A}\rightarrow \mathcal{M}$ satisfy the identity
 $$f(A)+A^{n}g(A^{-1})=0$$
 for each invertible element $A$ in $\mathcal{A}$, then $f=g=0.$
\end{lem}
\begin{proof}
For each positive integer $t$ and each invertible element $A\in \mathcal{A}$, by assumption, we have
 $$f(tA)+(tA)^{n}g((tA)^{-1})=0,$$
which means  $$t^{n-1}A^{n}g(A^{-1})+tf(A)=0.$$
Since the above equation holds for every positive integer $t$,
the coefficients of the above polynomial with respect to $t$  are equal to 0.
It follows from $n >2$ that $A^{n}g(A^{-1})=f(A)=0$.
Thus $g(A)=f(A)=0$ for each invertible element $A\in \mathcal{A}$.
On account of Remark \ref{rem1},  $g(T)=f(T)=0$ for each $T\in \mathcal{A}$.
\end{proof}

The following is a case of the identity (\ref{4}) for the integer $n>2$.
\begin{thm}
 Let $n >2$ be an integer.
 Suppose that $N$ is an invertible element in $\mathcal{A}$, and $M$ is a fixed element in  $\mathcal{M}$. 
If two additive mappings $f$, $g:\mathcal{A}\rightarrow \mathcal{M}$ satisfy the identity
 $$f(A)+A^{n}g(B)=M$$
whenever $A,B\in\mathcal{A}$ with $AB=N$, then $f=g=0.$
\end{thm}
\begin{proof}
For each positive integer $t$ and each invertible element $A\in \mathcal{A}$, it follows from $(tA)[(tA)^{-1})N]=N$ that 
 $$f(tA)+(tA)^{n}g((tA)^{-1}N)=M,$$
i.e., we have   $t^{n-1}A^{n}g(A^{-1}N)+tf(A)-M=0$ for each  positive integer $t$.
Because of $n >2$, we have $A^{n}g(A^{-1})=f(A)=M=0$ for each invertible element $A\in \mathcal{A}$.
Therefore $g(T)=f(T)=0$ for each $T\in \mathcal{A}$.
\end{proof}


\subsection*{Declarations}
The authors declare that they have no conflict of interest.
No data is used for the research described in the article.


\begin{thebibliography}{1}


\bibitem{ADL} An, G., Ding, Y., Li, J.: Characterizations of Jordan left derivations on some algebras. Banach J. Math. Anal. \textbf{10}(3), 466--481 (2016)

\bibitem{B1993} Bre\v{s}ar, M.: Centralizing mappings and derivations in prime rings. J. Algebra \textbf{156}(2), 385--394 (1993)

\bibitem{B1995} Bre\v{s}ar, M.: On generalized biderivations and related maps. J. Algebra \textbf{172}(3), 764--786 (1995)


\bibitem{B} Bre\v{s}ar, M.: Commuting maps: a survey. Taiwanese J. Math. \textbf{8}(3), 361--397 (2004)

\bibitem{BCM} Bre\v{s}ar, M., Chebotar, M.A., Martindale, W.S.: Functional identities. Birkhäuser, Basel (2007)
\bibitem{BV} Bre\v{s}ar, M., Vukman, J.: On left derivations and related mappings. Proc. Amer. Math. Soc. \textbf{110}(1), 7--16 (1990)

\bibitem{CM} Catalano, L., Merchán, T.: On rational functional identities. Comm. Algebra \textbf{52}(2),  717--722 (2024)





\bibitem{FJS} Ferreira, B.L.M., Julius, H., Smigly, D.:  Commuting maps and identities with inverses on alternative division rings. J. Algebra \textbf{638}, 488--505 (2024) 
 




\bibitem{LeeL} Lee, T.-K.,  Lin, J.-H.: Certain functional identities on division rings. J. Algebra \textbf{647}, 492--514 (2024)
   
\bibitem{LCW} Luo, Y., Chen, Q., Wang, Y.: On rational functional identities involving inverses on matrix rings. arXiv preprint,  arXiv:2310.07013.






\bibitem{N} Ng, C.T.: The equation $F(x)+M(x)G(1/x)=0$ and homogeneous biadditive forms. Linear Algebra Appl. \textbf{93}, 255--279 (1987)


\bibitem{P} Posner, E.C.: Derivations in prime rings. Proc. Amer. Math. Soc. \textbf{8}, 1093--1100 (1957)

\bibitem{V1987} Vukman, J.: A note on additive mappings in noncommutative fields. Bull. Austral. Math. Soc. \textbf{36}(3), 499--502 (1987)

\bibitem{V} Vukman, J.: On left Jordan derivations of rings and Banach algebras. Aequationes Math. \textbf{75}(3), 260--266 (2008)


\bibitem{XW} Xiao, Z., Wei, F.: Commuting mappings of generalized matrix algebras. Linear Algebra Appl. \textbf{433}(11--12), 2178--2197 (2010) 
\end{thebibliography}
\end{document}